\documentclass[11pt,a4paper,onehalfspacing]{amsart}
\usepackage{amsaddr}

\usepackage[utf8]{inputenc} 
\usepackage{ae}

\usepackage[english]{babel}
\usepackage[T1]{fontenc} 
\usepackage{amsmath, amssymb,amsthm}
\usepackage{graphicx}

\usepackage{todonotes}

\usepackage{subcaption}

\usepackage{hyperref} 

\usepackage{tikz-cd} 

\usepackage{setspace} 
\usepackage{microtype} 
\usepackage{enumerate} 
\usepackage{csquotes} 

\usepackage{mathtools}

\numberwithin{equation}{section} 

\newcommand{\NN}{\mathbb N}
\newcommand{\RR}{\mathbb R}

\newcommand{\dd}{\mathop{}\!\mathrm{d}} 

\DeclareMathOperator{\id}{id}

\newcommand{\set}[2][]{  
	\ensuremath{\left\{\hspace{0.1em}  %
	\if\relax\detokenize{#1}\relax
		{#2} 
		\else
		{#1}\,\middle\vert\,{#2}
	\fi
	\hspace{0.1em} \right\}}
}

\newcommand{\quotient}[2]{\left.\raisebox{.2em}{$#1$}\middle/\raisebox{-.2em}{$#2$}\right.}

\usepackage[citestyle=numeric-comp,bibstyle=numeric-comp,backend=bibtex,sorting=none]{biblatex}

\newtheorem{theorem}{Theorem}  
\newtheorem{corollary}[theorem]{Corollary}  
		
		\newtheorem{thm}{Theorem}[section]
		
		\newtheorem{lem}[thm]{Lemma}
		\newtheorem{cor}[thm]{Corollary}
		
		\newtheorem{prop}[thm]{Proposition}
	\theoremstyle{definition}	
		
		\newtheorem{defn}[thm]{Definition}
		\newtheorem{example}[thm]{Example}

\author[M.~G\"{u}nther]{Martin~G\"{u}nther$^\ast$}
\address[M.~G\"{U}nther]{Institute of Algebra and Geometry, Karlsruhe Institute of Technology (KIT), Karlsruhe, Germany.}
\email{\href{mailto:martin.guenther@kit.edu}{martin.guenther@kit.edu}}

\begin{document}


\title[categorical structure of timelike homotopy classes]{On the categorical and topological structure of timelike and causal homotopy classes of paths in smooth spacetimes}
\date{July 29, 2020}



	\begin{abstract}
	
	For a smooth spacetime $X$, based on the timelike homotopy classes of its timelike paths, we define a topology on $X$ that refines the Alexandrov topology and always coincides with the manifold topology.
	
	The space of timelike or causal homotopy classes forms a semicategory or a category, respectively. We show that either of these algebraic structures encodes enough information to reconstruct the topology and conformal structure of $X$.
	Furthermore, the space of timelike homotopy classes carries a natural topology that we prove to be locally euclidean but, in general, not Hausdorff.
	
	The presented results do not require any causality conditions on $X$ and do also hold under weaker regularity assumptions.

	\end{abstract}
	
\maketitle	

\section{Main results}

For simplicity, $(X,g)$ will be a smooth spacetime throughout this paper, i.e.\ a smooth time-oriented Lorentzian manifold.
Nevertheless, all of the following results are true for a broader class of manifolds, namely $\mathcal C^1$-spacetimes that satisfy lemma~\ref{lem:simple nbhds}, which is a local condition on the topology and chronological structure.

Let $P(X)$ be the space of paths in $X$, i.e.\ continuous maps $[0,1]\to X$, and $P^\mathrm{t}(X) \subseteq P^\mathrm{c}(X) \subseteq P(X)$ the spaces of continuous future-directed timelike or causal paths in $X$.
A juxtaposed pair of points $x,y\in X$ shall restrict these sets to only paths from $x$ to $y$.
We call two paths $c_0, c_1 \in P^{\mathrm{t/c}}(X)(x,y)$
\emph{causally/timelike homotopic relative to their endpoints}, or, for simplicity, just \emph{homotopic}, if there is a continuous map 
\[
H\colon[0,1]^2 \to X, \quad (s,t)\mapsto c_s(t) \quad
\]
with $c_s\in P^{\mathrm{t/c}}(X)(x,y)$ for all $s\in [0,1]$. 

It is well known (see e.g.\ \cite[Proposition 3.11.]{GlobalLorentzianGeometry}) that the \emph{chronological diamonds}
\begin{align*}
I_X(x,y) \coloneqq{}& \set[z\in X]{x\ll z \ll y} \\
={}& \set[z\in X]{\exists a\in P^\mathrm{t}(X)(x,z), b\in P^\mathrm{t}(X)(z,y)}
\end{align*}
for all $x,y\in X$ constitute a basis of the \emph{Alexandrov topology} on $X$, which is strictly coarser than the manifold topology on $X$ unless $X$ is strongly causal.

Let $ba$ be the chronological path obtained by concatenating%
\footnote{Note that the order $ba$ (rather than $ab$) is analogous to the composition of functions, where $g\circ f$ means \enquote{first $f$ then $g$}. This is also the convention used for composition of morphisms in a (semi-)category.}
$a$ and $b$
and reparametrizing linearly to the domain $[0,1]$.
By restricting $ba$ to a single homotopy class, we get a refinement of the above statement:

\begin{defn} \label{def:diamonds}
Let $X$ be a spacetime. For a path $c\in P^\mathrm{t}(X)(x,y)$, let
\begin{align*}
I_X([c]) &\coloneqq \set[z\in X]{\exists a\in P^\mathrm{t}(X)(x,z), b\in P^\mathrm{t}(X)(z,y) : ba\sim_\mathrm{t} c}, \\
\shortintertext{for a path $c\in P^\mathrm{c}(X)(x,y)$, let} 
J_X([c]) &\coloneqq \set[z\in X]{\exists a\in P^\mathrm{c}(X)(x,z), b\in P^\mathrm{c}(X)(z,y) : ba\sim_\mathrm{c} c},
\end{align*}
where $\sim_\mathrm{t}$ and $\sim_\mathrm{c}$ denote timelike and causal homotopy relative to the endpoints.
\end{defn}

\begin{theorem}\label{thm:Alexandrov-like}
In a spacetime $X$, the sets $I_X([c])$ for all $c\in P^\mathrm{t}(X)$ form a basis of the manifold topology of $X$.

There is a neighborhood basis (in the manifold topology) of every point in $X$ consisting of sets $J_X([c])$ for suitable $c\in P^\mathrm{c}(X)$. 
\end{theorem}

Note that no causality conditions are assumed in this theorem. We could even generalize the setting to non-time-orientable Lorentzian manifolds, by observing that theorem~\ref{thm:Alexandrov-like} provides a basis of the time-oriented double cover, which projects down to a basis of the manifold topology of $X$. 

The notation indicates that $I_X([c])$ or $J_X([c])$ only depends on the timelike or causal homotopy class $[c]$ of the path $c$, respectively. 
As with the usual chronological and causal diamonds, the sets $I_X([c])$ are open but the sets $J_X([c])$ are, in general, neither open nor closed. The latter may even have empty interior, if $c$ is a lightlike geodesic.

We now consider the quotient spaces $\Pi^{\mathrm{t}}(X)$ and $\Pi^{\mathrm{c}}(X)$, defined by
\begin{align}
\Pi^{\mathrm{t/c}}(X)(x,y) & \coloneqq \quotient{P^\mathrm{t/c}(X)(x,y)}{\sim_\mathrm{t/c}},
 \label{eqn:morphism_sets_1}\\
\Pi^{\mathrm{t/c}}(X)\phantom{(x,y)} & \coloneqq \quotient{P^\mathrm{t/c}(X)}{\sim_\mathrm{t/c}}.
\label{eqn:morphism_sets_2}
\end{align}

Since a strictly monotonic reparametrisation of a path does not change its homotopy class, the operation $[b][a]\coloneqq [ba]$ is well defined and associative. 
With this operation, $\Pi^{\mathrm{t/c}}(X)$ becomes a semicategory with object set $X$ and morphism sets $\Pi^{\mathrm{t/c}}(X)(x,y)$. 
Doing the same construction without restriction to causal or timelike paths yields the well-known fundamental groupoid $\Pi(X)$, see \cite{TopologyAndGroupoids}.

As concatenating any path with a constant path doesn't change the causal homotopy class, $\Pi^\mathrm{c}(X)$ is actually a category with identity morphisms.
On the other hand, corollary~\ref{cor:no inverses} will show that $\Pi^\mathrm{t}(X)$ is a genuine semicategory without any identity morphisms, essentially because constant paths are not timelike.

The following theorem and its corollary show that the isomorphism type of the category $\Pi^{\mathrm{c}}(X)$ or the semicategory $\Pi^{\mathrm{t}}(X)$ fully encodes the topology and the conformal structure of $X$:

\begin{theorem} \label{thm:conformal structure}
Let $(X,g)$, $(Y,h)$ be spacetimes. Then, any continuous map $f\colon X \to Y$ that maps future directed causal paths in $X$ to such in $Y$ induces a functor $f_\ast\colon \Pi^{\mathrm{c}}(X) \to \Pi^{\mathrm{c}}(Y)$.
Furthermore, for a functor $F\colon \Pi^{\mathrm{c}}(X) \to \Pi^{\mathrm{c}}(Y)$, the following statements are equivalent:
\begin{enumerate}[i)]
\item $F$ is an isomorphism, i.e.\ it has an inverse functor $F^{-1}\colon \Pi^{\mathrm{c}}(Y) \!\to\! \Pi^{\mathrm{c}}(X)$.
\item $F$ is induced by a homeomorphism $f\colon X\to Y$ such that both $f$ and $f^{-1}$ map future directed causal paths to future directed causal paths.
\end{enumerate}

The same is true if \enquote{causal} is replaced by \enquote{timelike} and $\Pi^c(X)$ is replaced by $\Pi^t(X)$.
\end{theorem}

Note that in ii), the map $f$ is the object map of $F$. If the functor $F$ is not an isomorphism, its object map will, in general, not be continuous.
Furthermore, theorem~\ref{thm:conformal structure} becomes wrong if \enquote{homeomorphism} is replaced by \enquote{diffeomorphism}. For example, the map \[ f\colon (\RR^2,-\dd x_1\dd x_2)\to (\RR^2,-\dd x_1\dd x_2), \quad (x_1,x_2)\mapsto(x_1\!^3,x_2\!^3) \] is not a diffeomorphism, but induces an isomorphism $F=f_*$.
But if we restrict $f$ to be a diffeomorphism (and especially in the case $f=\id$) a more specific conclusion can be made:

\begin{corollary}
\label{cor:conformal structure}
Let $(X,g)$ and $(Y,h)$ be two spacetimes and $F\colon \Pi^{\mathrm{c}}(X) \to \Pi^{\mathrm{c}}(Y)$ or $F\colon \Pi^{\mathrm{t}}(X) \to \Pi^{\mathrm{t}}(Y)$ an isomorphism of (semi-)categories. If the object map of $F$ is a diffeomorphism, then it is a conformal diffeomorphism.
\end{corollary}

For the following results, the reader may forget about the semicategory-structure of $\Pi^{\mathrm{t}}(X)$. Instead, we endow $P(X)$ with the compact-open topology, $P^\mathrm{t}(X)$ with the subset topology and $\Pi^{\mathrm{t}}(X)$ with the resulting quotient topology, as in \eqref{eqn:morphism_sets_1}.

For reference, the following lemma explains the topology of the fundamental groupoid $\Pi(X)$ (without restriction to timelike paths), as well as giving a construction of the universal covering of $X$:
\begin{lem}\label{lem:Topology of fundamental groupoid}
For a path connected, locally path connected and semilocally simply connected topological space $X$, the following statements are true:
\begin{enumerate}[i)]
\item 
The start/endpoint map $(s,e)\colon \Pi(X) \to X\times X$ is a covering map.
\item Let $\tilde X \coloneqq e^{-1}(x_0)$. Then $s\vert_{\tilde X}\colon \tilde X \to X$ is a universal covering.
\item The fibers of $(s,e)$, i.e.\ the sets $\Pi(X)(x,y)$ are discrete subsets of $\Pi(X)$.
\item
If $X$ is an $n$-dimensional manifold, $\Pi(X)$ is a $2n$-dimensional manifold. 
\end{enumerate}
\end{lem}

The proof is an exercise in covering theory and is included in the appendix.
In the case of $\Pi^{\mathrm{t}}(X)$, the situation is similar, but only a weaker conclusion can be made. The start/endpoint map cannot be a covering, since the cardinality of the fibers $\Pi^{\mathrm{t}}(X)(x,y)$ depends on the points $x,y\in X$.

\goodbreak

\begin{theorem}
\label{thm:Topology of Pchron}
Let $X$ be an $n$-dimensional spacetime. Then
\begin{enumerate}[i)]
\item $(s,e)\colon \Pi^{\mathrm{t}}(X) \to X\times X$ is a local homeomorphism. 
\item the sets $\Pi^{\mathrm{t}}(X)(x,y)$ are discrete subsets of $\Pi^{\mathrm{t}}(X)$.
\item $\Pi^{\mathrm{t}}(X)$ is locally homeomorphic to $\RR^{2n}$ but not necessarily Hausdorff.
\end{enumerate}
\end{theorem}

During the proof of this theorem, we will construct a purely algebraic description of the topology on $\Pi^{\mathrm{t}}(X)$. 
Example~\ref{ex:Non-Hausdorff} will show that it can indeed be non-Hausdorff, even under strong conditions on the spacetime $X$; in the example, $X$ is globally hyperbolic with compact, simply connected Cauchy surfaces.
The example will also show that theorem~\ref{thm:Topology of Pchron} becomes wrong if $\Pi^{\mathrm{t}}(X)$ is replaced by $\Pi^{\mathrm{c}}(X)$.

\section{Proof of theorem~\ref{thm:Alexandrov-like}}

For the sake of simplicity, the statements and proofs in this section are only presented in their timelike variant. If not stated otherwise, the statements and their proofs remain true if \enquote{timelike} is replaced by \enquote{causal}, $\Pi^\mathrm{t}$ is replaced by $\Pi^\mathrm{c}$ and $I_X$ is replaced by $J_X$.

\goodbreak
Let $X$ be a smooth spacetime.
The proofs of the main theorems all hinge on the following lemma:
\begin{lem}[timelike variant]\label{lem:simple nbhds} 
Every point $p\in X$ has arbitrarily small open neighborhoods $U$ such that for each $x,y\in U$, there is at most one timelike homotopy class in $\Pi^{\mathrm{t}}(U)(x,y)$.
Furthermore, $U$ can be chosen such that it is strongly causal as a Lorentzian submanifold.
\end{lem}

\begin{cor}[timelike variant]\label{corr:simple nbhds}
Let $U\subseteq X$ be as in lemma~\ref{lem:simple nbhds}.
\begin{itemize}
\item
The functor $\Pi^t(U) \to \Pi^t(X)$ induced by the inclusion $U \hookrightarrow X$ 
is injective on both the objects and morphisms. Therefore, we will always interpret $\Pi^t(U)$ as sub-semicategory of $\Pi^t(X)$.
\item
Every diagram in $\Pi^t(X)$ that only involves morphisms from $\Pi^t(U)$ is a commutative diagram.
\end{itemize}
\end{cor}

\begin{proof}[Proof of lemma~\ref{lem:simple nbhds}]
We show that every simply convex normal neighborhood $U$ of $p$ has the desired property.
Any points $x,y\in U$ can be joined by a unique geodesic $\gamma_{xy}$ in $U$, which depends continuously on the endpoints.
If there is any timelike path in $U$ from $x$ to $y$, then the geodesic $\gamma_{xy}$ is also timelike \cite[Proposition 4.5.1]{HawkingEllis}.
Hence, for any path $c \in P^\mathrm{t}(X)(x,y)$, the map ${H\colon [0,1] \to P^\mathrm{t}(X)}$, ${s \mapsto c_s \coloneqq c\vert_{[s,1]} \cdot \gamma_{c(0)c(s)}}$ is a timelike homotopy from $c=c_0$ to $\gamma_{xy} = c_1$. This shows that $\Pi^\mathrm{t}(U)(x,y)$ is either empty or equal to $\set{[\gamma_{xy}]}$.

We finish the proof by showing that there are arbitrarily small neighborhoods of $p$ which are both simply convex and strongly causal. First fix an arbitrary chart $\phi\colon \tilde U\to \tilde V\subseteq \RR^n$ around $p$, and a constant Lorentzian metric $h$ on $\RR^n$ such that $h > \phi_*g_p$, i.e. $h$ has a larger light cone than $\phi_\ast g$ at the point $\phi(p)$.
As this is an open condition, we can choose $\tilde U$ small enough such that $h> \phi_*g_x$ holds for all points $x\in\tilde U$.
Since $(\RR^n,h)$ is strongly causal, $(\tilde V,h)$ and $(\tilde V,\phi_\ast g) \cong (\tilde U, g)$ are strongly causal submanifolds of $\RR^n$ or $X$, respectively.

Now, any simply convex normal neighborhood $U\subseteq\tilde U$ of $p$, is also simply convex in $X$ and strongly causal as as submanifold of $X$.
\end{proof}

The neighborhoods $U$ in lemma~\ref{lem:simple nbhds} could be called \emph{timelike simply connected}, in analogy to simply connected fundamental groupoids \cite[p.\ 213]{TopologyAndGroupoids}. Nonexistence of closed (or almost closed) timelike or causal curves in $U$ is necessary, but not sufficient for this property.

For the rest of the paper, we will treat the statement of lemma~\ref{lem:simple nbhds} like an axiom, and we will require neither smoothness of $X$ nor the existence of geodesics anymore. Therefore, all the main theorems of this paper hold true for any $\mathcal C^1$-spacetime $X$ in which the statement of lemma~\ref{lem:simple nbhds} holds. 

\goodbreak
We break up the proof of theorem~\ref{thm:Alexandrov-like} into two further lemmata:

\begin{lem}[timelike variant]\label{lem:diamonds in scnn}
Let $U\subseteq X$ be a neighborhood as in lemma~\ref{lem:simple nbhds} (e.g.\ a simply convex normal neighborhood), $x,y\in U$ and $c\in P^\mathrm{t}(U)(x,y)$.
We denote by $I_U(x,y)$ the chronological diamond in $U$, as a Lorentzian submanifold.

Then we have $I_U(x,y) = I_U([c]) \subseteq I_X([c])$. Equality holds if $\overline{I_U(x,y)}\subseteq U$, where the overline denotes the closure in $X$.
\end{lem}

\begin{proof}
$I_U([c])=I_U(x,y)$ follows directly from definition~\ref{def:diamonds} and lemma~\ref{lem:simple nbhds}, since all timelike paths in $U$ from $x$ to $y$ are in the same timelike homotopy class $[c]\in \Pi^{\mathrm{t}}(U)(x,y)$. Any timelike homotopy in $U$ is also a timelike homotopy in $X$, therefore $I_U([c]) \subseteq I_X([c])$.

If we now assume $I_U([c]) \neq I_X([c])$ there must be a timelike homotopy $H\colon [0,1]^2\to X,(s,t)\mapsto c_s(t)$ starting in $c_0=c$, that is not a timelike homotopy in $U$. Consider the sets
\begin{align*}
S_1&=\set[{s \in [0,1]}]{ \forall t\in[0,1]: c_s(t)\in U}\\
S_2&=\set[{s \in [0,1]}]{\forall t\in[0,1]: c_s(t)\in \overline{I_U(x,y)}}.
\end{align*}

By assumption, $0\in S_1 \subsetneq [0,1]$. Since $[0,1]$ is compact, $U$ is open and $\overline{I_U(x,y)}$ is closed, we see that $S_1$ is open in $[0,1]$ and $S_2$ is closed.

The images of all paths in $P^\mathrm{t}(U)(x,y)$ are contained in $\overline{I_U(x,y)}$, hence $S_1\subseteq S_2$. The assumption $\overline{I_U(x,y)}\subseteq U$ then implies $S_2\subseteq S_1$, hence $S_2=S_1$, which leads to a contradiction.
\end{proof}

Note that the assumption $\overline{I_U(x,y)}\subseteq U$ in the above lemma is nontrivial; it basically ensures that $I_U(x,y)$ does not touch the boundary of $U$.

\begin{lem}\label{lem:Neighborhood basis}
For $c\in P^\mathrm{t}(X)$, $t\in (0,1)$ and $0< \varepsilon < \min(t,1-t)$, let $c_\varepsilon$ be the reparametrization of the path $c\vert_{[t-\varepsilon,t+\varepsilon]}$ to the unit interval.\\
Then, the sets $I_X([c_\varepsilon])$ form an open neighborhood basis of $c(t)$ in the manifold topology. The sets $J_X([c_\varepsilon])$ also form a (generally not open) neighborhood basis of $c(t)$.
\end{lem}

\begin{proof} Let $U$ be a neighborhood of $p\coloneqq c(t)$ as in lemma~\ref{lem:simple nbhds}. 
Since $X$ is a manifold, there is some smaller neighborhood $V$ of $p$ with $\overline V\subseteq U$. By strong causality of $U$, the sets $I_U(c(t-\varepsilon),c(t+\varepsilon))$, form a neighborhood basis of $p$, so for small enough $\varepsilon>0$, we have $I_U(c(t-\varepsilon),c(t+\varepsilon))\subseteq V$, hence $\overline{I_U(c(t-\varepsilon),c(t+\varepsilon))}\subseteq \overline V \subseteq U$ and $I_X([c_\varepsilon]) = I_U(c(t-\varepsilon),c(t+\varepsilon))$, by lemma~\ref{lem:diamonds in scnn}. This shows that the sets $I_X([c_\varepsilon])$ also form a neighborhood basis of $c(t)$.

For $0<\varepsilon' < \varepsilon$, we have the inclusions
\[ c(t)\in I_U(c(t-\varepsilon'),c(t+\varepsilon'))\subseteq J_U(c(t-\varepsilon'),c(t+\varepsilon')) \subseteq I_U(c(t-\varepsilon),c(t+\varepsilon)) \]
from which we see that the sets $J_U(c(t-\varepsilon'),c(t+\varepsilon'))$ also form a neighborhood basis of $c(t)$. Furthermore, this implies $\overline{J_U(c(t-\varepsilon'),c(t+\varepsilon'))}\subseteq U$, hence $J_X([c_\varepsilon']) = J_U(c(t-\varepsilon'),c(t+\varepsilon'))$, by the causal variant of lemma~\ref{lem:diamonds in scnn}. Therefore, the sets $J_X([c_\varepsilon])$ also form a neighborhood basis of $c(t)$.
\end{proof}

Note that, even when we only consider causal diamonds, we demand $c$ to be timelike, as $J_U(c(t-\varepsilon),c(t+\varepsilon))$ is not a neighborhood of $c(t)$ if $c$ is a lightlike geodesic.

Theorem~\ref{thm:Alexandrov-like} follows directly from lemma~\ref{lem:Neighborhood basis}: Since there is some timelike curve through any point, lemma~\ref{lem:Neighborhood basis} provides a neighborhood basis of any point. In the timelike case, these neighborhoods are open and therefore form an actual basis.

For the next corollary, remember that we only consider timelike or causal homotopies with fixed start- and endpoints.

\begin{cor}\leavevmode
\label{cor:no inverses}
\begin{enumerate}[i)]
\item
The identity morphisms in $\Pi^{\mathrm{c}}(X)$ are only represented by constant paths.
\item
No non-identity morphism in $\Pi^{\mathrm{c}}(X)$ has an inverse morphism,\\
therefore, $\Pi^{\mathrm{c}}(X)$ is a category, but not a groupoid.
\item
There is no identity morphism in $\Pi^{\mathrm{t}}(X)$,\\
therefore, $\Pi^{\mathrm{t}}(X)$ is a semicategory, but not a category.
\end{enumerate}
\end{cor}
\begin{proof}

Suppose that some path $e$ represents an identity morphism $[e]$ in $\Pi^\mathrm{t}(X)(x,x)$ or $\Pi^\mathrm{c}(X)(x,x)$. Let $c$ be a timelike path such that $c(t)=x$, and choose $\varepsilon$ as in lemma~\ref{lem:Neighborhood basis}. Because $[e]$ is an identity, we have \[ [c_\varepsilon]=\left[c\vert_{[t,t+\varepsilon]} \right] \left[ c\vert_{[t-\varepsilon,t]} \right]=\left[c\vert_{[t,t+\varepsilon]} \right][e] \left[ c\vert_{[t-\varepsilon,t]} \right]\]
which implies $I_X([e])\subseteq I_X([c_\varepsilon])$ or $J_X([e])\subseteq J_X([c_\varepsilon])$, respectively. As, by lemma~\ref{lem:Neighborhood basis}, both the families $I_X([c_\varepsilon])$ and $J_X([c_\varepsilon])$ form neighborhood bases of $x$, we see $I_X([e]) \subseteq \{x\}$ or $J_X([e]) \subseteq \{x\}$, respectively, for $\varepsilon \to 0$. 
This can only be true if $e$ is the constant path from $x$ to $x$, which is indeed a causal but not a timelike path, so we proved i) and iii).

For $a,b\in P^\mathrm{c}(X)$ such that $[a]$ and $[b]$ are inverse to each other, the concatenation $ab$ is constant by i). Therefore, $a$ and $b$ already represent identity morphisms, which implies ii).
\end{proof}

\section{Proof of theorem~\ref{thm:conformal structure}}

We will prove the timelike variant of theorem~\ref{thm:conformal structure}. The proof for the causal variant is completely analogous.

Let $f\colon X \to Y$ be a continuous map such that for any $c\in P^\mathrm{t}(X)$, the push-forward $f_\ast(c) \coloneqq f\circ c$ is in $P^\mathrm{t}(Y)$. From the definitions, it is clear that the push-forward of a timelike homotopy in $X$ is a timelike homotopy in $Y$, hence $f_\ast$ is well defined on timelike homotopy classes and induces a functor $f_\ast\colon \Pi^{\mathrm{t}}(X)\to \Pi^{\mathrm{t}}(Y)$.

ii) $\implies$ i): If ii) is true, $(f^{-1})_{\ast}\colon \Pi^{\mathrm{t}}(Y)\to \Pi^{\mathrm{t}}(X)$ is also a functor. It is inverse to $f_\ast$, so $f_\ast$ is an isomorphism of semicategories.

i) $\implies$ ii): Let $F\colon \Pi^{\mathrm{t}}(X) \to \Pi^{\mathrm{t}}(Y)$ be an isomorphism of semi\-cate\-go\-ries and $f\colon X\to Y$ its map on the object sets. We want to show that $F=f_\ast$.

First, we show that $f$ is a homeomorphism. By definition~\ref{def:diamonds}, for any $z\in I_X([c])$, there exist $[a]\in \Pi^\mathrm{t}(x,z)$, $[b]\in \Pi^\mathrm{t}(z,y)$ with $[b][a]=[c]$.
This implies $F([b])F([a])=F([c])$, hence $f(z)\in I_Y(F([c]))$.
Applying the same argument to $F^{-1}$, one also sees that $z \in I_Y(F([c]))$ implies $f^{-1}(z)\in I_X([c])$, hence $f(I_X([c])) = I_Y(F([c]))$.
Therefore both $f$ and $f^{-1}$ map the bases, given by theorem~\ref{thm:Alexandrov-like}, of the topologies on $X$ and $Y$ to one another, so $f$ is a homeomorphism.

Now we show that for any $c\in P^\mathrm{t}(X)$, the continuous path $f\circ c$ is timelike in $Y$. At first, consider \emph{short enough} paths, such that there is a simply convex normal neighborhood $U \subseteq Y$ with $\overline{I_X([c])} \subseteq f^{-1}(U)$. For any $0\leq t_0 < t_1 \leq 1$, the class $F\left(\left[ c\vert_{[t_0,t_1]} \right]\right)$ is nonempty and consists of classes of timelike paths in
\[ \overline{I_Y\left(F\left(\left[ c\vert_{[t_0,t_1]} \right]\right)\right)} = f\left(\overline{I_X([c\vert_{[t_0,t_1]}])}\right)
\subseteq f\left(\overline{I_X([c])}\right) \subseteq U \]
from $f(c(t_0))$ to $f(c(t_1))$, hence $f(c(t_0))\ll_U f(c(t_1))$. This shows that the path $f \circ c$ is indeed timelike in $U$, hence in $Y$. Because timelike homotopy classes of paths in $U$ are uniquely determined by their start and endpoint (see lemma \ref{lem:simple nbhds}), this also implies $f_*([c])=[f\circ c] = F([c])$.

By compactness, every timelike path is a finite concatenation of short enough paths (cf.\ lemma~\ref{lem:Neighborhood basis}). Therefore the functors $f_\ast$ and $F$ actually coincide for all timelike paths, which finishes the proof.
\qed

\section{Proof of corollary~\ref{cor:conformal structure}}

By theorem~\ref{thm:conformal structure}, and by the assumption that the object map $f$ is a diffeomorphism, $f$ maps future directed differentiable timelike/causal curves in $(X,g)$ to such in $(Y,h)$. Consequently, the differential of $f$ maps the light cones in $TX$ to those in $TY$ and preserves the time orientation. It is well-known (see \cite[Lemma 2.1.\ et seq.]{GlobalLorentzianGeometry}) that in this case, $f$ is a conformal diffeomorphism. \qed

\section{Proof of theorem~\ref{thm:Topology of Pchron}}

In this section we will only consider the space $\Pi^\mathrm{t}(X)$ of timelike homotopy classes, as we rely on the openness of the chronological relation $\ll$ and timelike diamonds.

By definition, a basis of the compact-open topology on $P^{\mathrm{t}}(X)$ is given by all finite intersections of sets of the form
\[ \Omega(K,O) \coloneqq \set[c\in P^t(X)]{c(K)\subseteq O}, \]
where $K\subseteq[0,1]$ is compact and $O\subseteq X$ is open.
Remember that according to equation \eqref{eqn:morphism_sets_1}, $\Pi^{\mathrm{t}}(X)$ carries the corresponding quotient topology.

As a homeomorphism is necessarily open and continuous, we first show the following:

\begin{prop} \label{prop:(s,e) is open}
Let $(X,g)$ be a $\mathcal C^1$ spacetime. Then the start/endpoint-map $(s,e) \colon \Pi^{\mathrm{t}}(X) \to X\times X$ is open (i.e.\ images of open sets are open) and continuous (i.e.\ preimages of open sets are open).
\end{prop}
\begin{proof}
Because the start/endpoint map of $\Pi^{\mathrm{t}}(X)$ is induced by the one of the path space $P^{\mathrm{t}}(X)$, we only need to show that the latter is open and continuous.

If $O_1, O_2 \subseteq X$ are open subsets, then the preimage
\[  (s,e)^{-1} (O_1\times O_2) = \Omega(\set0,O_1) \cap \Omega(\set1,O_2) \]
is open in $P^\mathrm{t}(X)$. As the topology on $X\times X$ is generated by products of open sets, this shows the continuity of $(s,e)$.

One can check that for any path $c\in P^\mathrm{t}(X)$, there is a neighborhood basis of $c$ consisting of open sets $\mathcal{U} \coloneqq \bigcap_{i=1}^{k} \Omega\bigl([t_i,t_{i+1}],O_i)$ where $k\in\NN$, $0=t_1<t_2<\dots<t_{k+1}=1$, and $O_i\subseteq X$ is an open neighborhood of $c([t_i,t_{i+1}])$. Any path in $\mathcal{U}$ is a concatenation of timelike paths in $O_1,\dots O_k$. This means that a path from $x$ to $y$ exists in $\mathcal{U}$ if and only if there are points $x_1,\dots,x_k$ such that  $x=x_1\ll_{O_1}x_2\ll_{O_2}\dots\ll_{O_k} x_{k+1} = y$. As the chronological relations $\ll_{O_i}$ are open, we see that $(s,e) (\mathcal{U})$ is open in $X\times X$.
\end{proof}

\begin{cor}
Let $U$ be a neighborhood as in lemma~\ref{lem:simple nbhds} in a $\mathcal C^1$-spacetime. Then the map 
\[ (s,e) \colon \Pi^{\mathrm{t}}(U) \to \set[(x,y)\in U\times U]{x\ll_U y} \]
is a homeomorphism.
\end{cor}
\begin{proof}
Lemma~\ref{lem:simple nbhds} implies that $(s,e)$ is bijective, and proposition~\ref{prop:(s,e) is open} implies that it continuous and open, hence its inverse is also continuous.
\end{proof}
Note that, in the special case of a simply convex normal neighborhood $U$, the proof of lemma~\ref{lem:simple nbhds} already shows that the map
\begin{align*}
\set[(x,y)\in U\times U]{x\ll_U y} &\to \Pi^{\mathrm{t}}(U) \\ \quad (x,y) &\mapsto [\gamma_{xy}]
\end{align*}
is a continuous inverse to $(s,e)$, because the geodesics $\gamma_{xy}$ depend continuously on the points $x,y$. 
The above corollary generalizes this observation to $\mathcal{C}^1$-spacetimes, where the existence of geodesics is not guaranteed.

In order to prove theorem~\ref{thm:Topology of Pchron} and to obtain a purely algebraic description of the topology on $\Pi^\mathrm{t}$, we will show that the sets in the following definition constitute a basis:

\begin{defn} \label{def:U([a],[b],[c])}
For any morphisms $w\overset{[a]}{\longrightarrow}x\overset{[b]}{\longrightarrow}y\overset{[c]}{\longrightarrow}z$ in $\Pi^{\mathrm{t}}(X)$, let $\mathcal{U}([a],[b],[c]) \subseteq \Pi^{\mathrm{t}}(X)$ be the set of morphisms $[d]$ for which there are additional morphisms $[a_1],[a_2],[c_1],[c_2]$ (dashed arrows) that make the following diagram commute:

\[
\begin{tikzcd}[row sep=small]
z \\
 & q \arrow[swap,dashed, near end]{ul}{[c_2]} \\
y \arrow{uu}{[c]}
	\arrow[swap,dashed, near start]{ur}{[c_1]} \\
\null \\
x \arrow{uu}{[b]} \\
& p \arrow[swap]{uuuu}{[d] \, \in\,  \mathcal{U}([a],[b],[c])}
	\arrow[swap,dashed, near end]{ul}{[a_2]} \\
w \arrow{uu}{[a]}
	\arrow[swap,dashed, near start]{ur}{[a_1]}
\end{tikzcd}
\]
\end{defn}

\begin{lem}
\label{lem:U([a],[b],[c]) open}
The sets $\mathcal{U}([a],[b],[c])$ are open in $\Pi^{\mathrm{t}}(X)$.
\end{lem}

\begin{proof}
Let $\pi\colon P^\mathrm{t}(X) \to \Pi^\mathrm{t}(X), c \mapsto [c]$ be the quotient map. We will show that $\pi^{-1}\bigl( \mathcal{U}([a],[b],[c]) \bigr)$ is open in $P^\mathrm{t}(X)$ by constructing an open neighborhood $\mathcal{U}_d\subseteq \pi^{-1}\bigl( \mathcal{U}([a],[b],[c]) \bigr)$ of any path $d\in \pi^{-1}\bigl( \mathcal{U}([a],[b],[c]) \bigr)$.

Starting with a diagram as in definition~\ref{def:U([a],[b],[c])}, we concatenate the paths $a_1$, $d$ and $c_2$ to a path $\tilde d\colon  [-1,2]\to X$ from $w$ to $z$ such that $\tilde d\vert_{[0,1]} = d$.

Now we choose $k\in \NN$, some neighborhoods $U_{1},\dots,U_{k-1}$ as in lemma~\ref{lem:simple nbhds}, and
\[
-1 < t_{-1}<t_0=0<t_1<\dots <t_k=1<t_{k+1}<2
\]
such that 
\begin{enumerate}[i)]
\item For every $i \in \set{1,\dots,k-1}$, the sets $\overline{I_X([\tilde d\vert_{[t_{i-2},t_{i}]}])}$
and $\overline{I_X([\tilde d\vert_{[t_{i},t_{i+2}]}])}$ lie inside $U_i$.
\item
There are paths $a_1',a_2',c_1',c_2'$ such that the diagrams in figures~\ref{fig:diagram ii 1}~and~\ref{fig:diagram ii 2} commute.
\end{enumerate}

We can achieve i) since, by compactness, the image of $\tilde d$ is covered by a finite number of such neighborhoods $U_i$, and the differences between $t_i$ and $t_{i\pm 2}$ can be made sufficiently small by choosing $k$ large enough (cf.\ the proof of lemma~\ref{lem:Neighborhood basis}).
Condition ii) can be achieved by choosing $t_{\pm 1}$ and $t_{k\pm 1}$ sufficiently close to $0$ or $1$, respectively.  

\begin{figure}[tbp]
\begin{minipage}{0.6\linewidth}
	\begin{subfigure}{0.49\linewidth}
		\[
		\begin{tikzcd}[row sep=small,column sep=normal]
		x \\
		& \tilde d(t_{1}) \arrow[dashed,swap]{ul}{[a_2']} \\
		& \tilde d(t_{0}) \arrow{u} \arrow[dashed,bend left=10,near start]{uul}{[a_2]} \\
		& \tilde d(t_{-1}) \arrow{u} \\
		w \arrow{uuuu}{[a]} \arrow[dashed,swap]{ur}{[a_1']} \arrow[dashed,bend left=10,near end]{uur}{[a_1]}
		\end{tikzcd}
		\]
		\caption{}\label{fig:diagram ii 1}
	\end{subfigure}
	\hfill
	\begin{subfigure}{0.49\linewidth}
		\[
		\begin{tikzcd}[row sep=small,column sep=normal]
		z \\
		& \tilde d(t_{k+1}) \arrow[dashed,swap]{ul}{[c_2']}  \\
		& \tilde d(t_{k}) \arrow{u} \arrow[dashed,bend left=10,near start]{uul}{[c_2]} \\
		& \tilde d(t_{k-1}) \arrow{u}	\\
		y \arrow{uuuu}{[c]} \arrow[dashed,swap]{ur}{[c_1']} \arrow[dashed,bend left=10,near end]{uur}{[c_1]}
		\end{tikzcd}
		\]
		\vspace{-3ex}
		\caption{}\label{fig:diagram ii 2}
	\end{subfigure} \\[5ex]
	\begin{subfigure}{0.49\linewidth}
		\[
		\begin{tikzcd}[row sep=small,column sep=small]
		 \tilde d(t_{i+1}) \\
		 \tilde d(t_{i}) \arrow{u} & d'(t_{i}) \arrow[dashdotted]{ul} \\
		 \tilde d(t_{i-1}) \arrow{u}\arrow[dashdotted]{ur} \\
		\end{tikzcd}
		\]
		\vspace{-3ex}
		\caption{}\label{fig:diagram through d'}
	\end{subfigure}
	\hfill
	\begin{subfigure}{0.49\linewidth}
		\[
		\begin{tikzcd}[row sep=small,column sep=small]
		& d'(t_{i+1}) \\
		\tilde d(t_{1}) \arrow[dashdotted]{ur} & d'(t_{i}) \arrow{u} \\
		& d'(t_{i-1}) \arrow{u}\arrow[dashdotted]{ul} \\
		\end{tikzcd}
		\]
		\vspace{-3ex}
		\caption{}\label{fig:diagram through d}
	\end{subfigure}
\end{minipage}
\hfill
\begin{minipage}{0.35\linewidth}
	\begin{subfigure}{\linewidth}
		\[
		\begin{tikzcd}[row sep=small,column sep=tiny]
		z \\
		& \tilde d(t_{k+1}) \arrow[dashed]{ul}  \\
		& \tilde d(t_{k}) \arrow{u} & d'(t_{k}) \arrow[dashdotted]{ul} \\
		& \tilde d(t_{k-1}) \arrow{u}\arrow[dashdotted]{ur} & d'(t_{k-1}) \arrow{u}\arrow[dashdotted]{ul}  	\\
		y \arrow{uuuu}{[c]} \arrow[dashed]{ur}  & ~~\vdots~~  \arrow{u}\arrow[dashdotted]{ur} & ~~\vdots~~  \arrow{u}\arrow[dashdotted]{ul} \\
		 & ~~\vdots~~ & ~~\vdots~~ \\
		x \arrow{uu}{[b]} & \tilde d(t_{2}) \arrow{u}\arrow[dashdotted]{ur} & d'(t_{2}) \arrow{u}\arrow[dashdotted]{ul}\\
		& \tilde d(t_{1}) \arrow{u}\arrow[dashdotted]{ur}\arrow[dashed]{ul} & d'(t_{1}) \arrow{u}\arrow[dashdotted]{ul} \\
		& \tilde d(t_{0}) \arrow{u}\arrow[dashdotted]{ur} & d'(t_{0}) \arrow{u}\arrow[dashdotted]{ul} \\
		& \tilde d(t_{-1}) \arrow{u}\arrow[dashdotted]{ur} 	\\
		w \arrow{uuuu}{[a]}
			\arrow[dashed]{ur}
		\end{tikzcd}
		\]
		\vspace{-3ex}
		\caption{}\label{fig:diagram large}
	\end{subfigure}
\end{minipage}
\caption{Commuting diagrams in $\Pi^{\mathrm{t}}(X)$ for the proof of lemma~\ref{lem:U([a],[b],[c]) open}.
Unlabeled vertical arrows are timelike homotopy classes of paths $\tilde d\vert_{[t_i,t_{i+1}]}$ or $d'\vert_{[t_i,t_{i+1}]}$, respectively.
}
\end{figure}
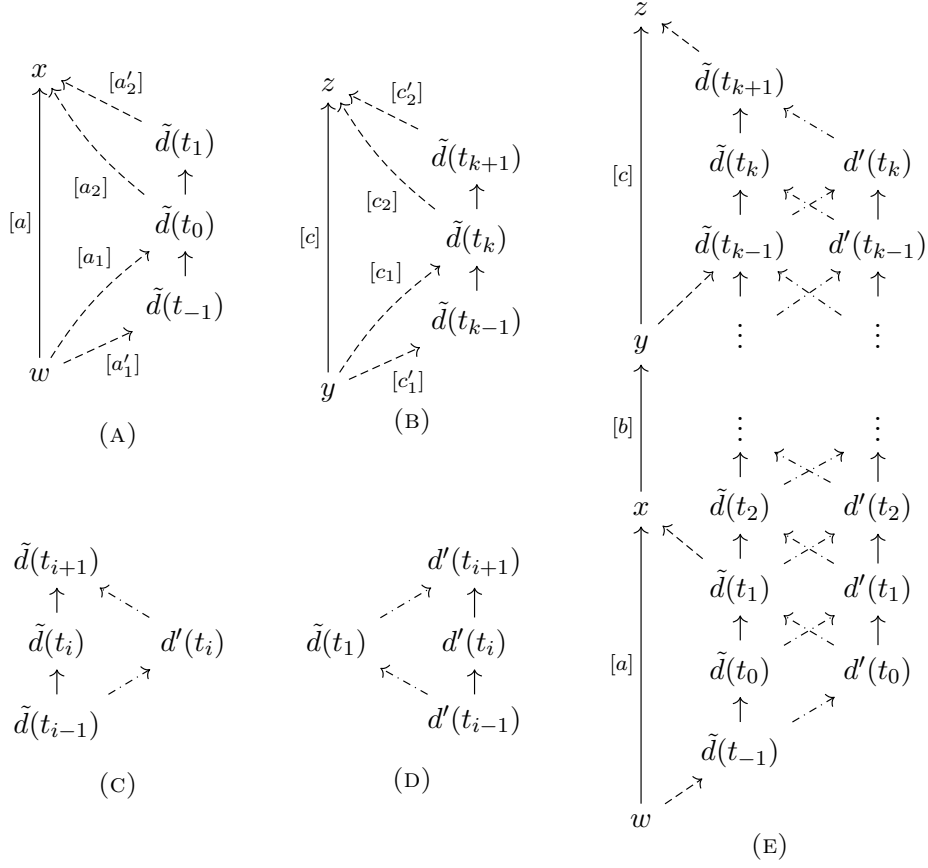

Now define the set
\begin{align*}
\mathcal{U}_d \coloneqq \Big\lbrace\,  d' \in P^\mathrm{t}(X) \, \big \vert \, 
\forall i\in\set{0,\dots,k} &: d'(\{t_i\}) \subseteq  I_X \left( [\tilde d\vert_{[t_{i-1},t_{i+1}]}] \right), \\
\forall i\in\set{1,\dots,k-1} &: d' \left( [t_{i-1},t_{i+1}] \right) \subseteq U_i
\, \Big\rbrace .
\end{align*}

By construction, $\mathcal{U}_d$ is open in the compact-open topology and contains the path $d$.
Now let $d' \in \mathcal{U}_d$. We need to show that $[d'] \in \mathcal{U}([a],[b],[c])$:
Since $d'(t_i) \in I_X \left( [\tilde d\vert_{[t_{i-1},t_{i+1}]}] \right)$ for every $i\in \set{0,\dots,k}$, there are morphisms (dash-dotted arrows) that make the diagram in figure~\ref{fig:diagram through d'} commute.

These morphisms are homotopy classes of paths in $U_{i\pm 1}$,by condition i). Therefore, all morphisms in the diagram in figure~\ref{fig:diagram through d} are homotopy classes of paths in $U_i$. Corollary~\ref{corr:simple nbhds} then implies that this diagram also commutes.

Taking everything together, we see that the diagram in figure~\ref{fig:diagram large} commutes, hence $[d']\in \mathcal{U}([a],[b],[c])$.
\end{proof}
\goodbreak

\begin{lem} 
\label{lem:(s,e) is local homeo}
Let $[a],[b],[c]$ as in definition~\ref{def:U([a],[b],[c])}. If $\overline{I_X([a])}\subseteq U_a$ and $\overline{I_X([c])} \subseteq U_c$ for some neighborhoods $U_a, U_c \subseteq X$ as in lemma~\ref{lem:simple nbhds}, the start/endpoint map maps $\mathcal{U}([a],[b],[c])$ homeomorphically onto $I_X([a]) \times I_X([c])$.
\end{lem}
\begin{proof}
By proposition~\ref{prop:(s,e) is open}, $(s,e)$ is open and continuous, and its restriction to the open set $\mathcal{U}([a],[b],[c])$ inherits these properties. We only need to show that $(s,e)$ maps $\mathcal{U}([a],[b],[c])$ bijectively onto $I_X([a]) \times I_X([c])$.

By definition~\ref{def:diamonds}, the points $p\in I_X([a])$, $q\in I_X([c])$ are exactly the points for which there are timelike homotopy classes $[a_1]\in \Pi^{\mathrm{t}}(X)(w,p)$, $[a_2]\in \Pi^{\mathrm{t}}(X)(p,x)$,$[c_1]\in \Pi^{\mathrm{t}}(X)(y,q)$ and $[c_2]\in \Pi^{\mathrm{t}}(X)(q,z)$ with $[a_2][a_1]=[a]$ and $[c_2][c_1]=[c]$. As the commuting diagram in definition~\ref{def:U([a],[b],[c])} is then completed by setting $[d] \coloneqq [c_1][b][a_2]$, we see that $(p,q)\in (s,e) \bigl( \mathcal{U}([a],[b],[c]) \bigr)$, hence $(s,e)$ maps $\mathcal{U}([a],[b],[c])$ surjectively onto $I_X([a]) \times I_X([c])$.

On the other hand, in a diagram as in definition~\ref{def:U([a],[b],[c])}, the representative paths $a_1$ and $a_2$ lie completely in $\overline{I([a])}\subseteq U_a$, hence their timelike homotopy classes are uniquely determined by the point $p$ (see lemma~\ref{lem:simple nbhds}). Analogously, the classes $[c_1]$ and $[c_2]$ are uniquely determined by $q$, hence the homotopy class $[d] = [c_1][b][a_2]$ is the unique preimage of $(p,q)$ in $\mathcal{U}([a],[b],[c])$. This shows that $(s,e)\vert_{\mathcal{U}([a],[b],[c])}$ is injective.
\end{proof}

Lemma~\ref{lem:(s,e) is local homeo} implies part i) of theorem~\ref{thm:Topology of Pchron}, which in turn implies the rest of theorem~\ref{thm:Topology of Pchron}. Since the sets $I_X([a])\times I_X([b])$ form a basis of $X\times X$, we also see that the sets $\mathcal{U}([a],[b],[c])$ form a basis of the topology of $\Pi^{\mathrm{t}}(X)$.

\section{An example}

\begin{example} \label{ex:Non-Hausdorff}
Let $X\coloneqq (\RR \times S, -\dd t^2 + g)$ be a Lorentzian product manifold, where $(S,g)$ is a closed Riemannian manifold, and fix $p,q\in S$ and $T\in \RR$.
We will use Morse theory (see e.g.\ \cite[chapter 10]{GlobalLorentzianGeometry} for a summary of Lorentzian Morse theory) to examine the homotopy type of the space of timelike or causal paths $P^{\mathrm{t/c}}(X)\bigl((0,p),(T,q) \bigr)$, as the connected components of this space are exactly the timelike or causal homotopy classes of paths.

Since $X$ is a product manifold, any path $c\in P^{\mathrm{t/c}}(X)\bigl((0,p),(T,q) \bigr)$ can be uniquely reparametrized to the form
\[ c \colon [0,1] \to \RR \times S, \quad t \mapsto \bigl(t\, T,\tilde c(t) \bigr), \]
where $\tilde c$ is a continuous path in $S$. On the other hand, given a path $\tilde c$ in $S$, the path $c$ is timelike if and only if $\tilde c$ satisfies 
\[ d_g(\tilde c(t_1),\tilde c(t_2)) < T\cdot(t_2-t_1) \text{ for all } 0\leq t_1<t_2\leq 1, \]
and causal if and only if it satisfies 
\[ d_g(\tilde c(t_1),\tilde c(t_2)) \leq T\cdot(t_2-t_1) \text{ for all } 0\leq t_1<t_2\leq 1, \]
i.e. if $\tilde c$ is $T$-Lipschitz.
Additionally, $c$ is a geodesic in $X$ if and only if $\tilde c$ is a geodesic in $S$. If this is the case, it is not hard to see that $\tilde c$ and $c$ have the same geodesic index.

By reparametrizing each path $\tilde c$ in $S$ proportionally to its arc length, we get a homotopy equivalence between $P^{\mathrm{t/c}}(X)\bigl((0,p),(T,q) \bigr)$ and the space of paths in $S$ of length $<T$ or $\leq T$, respectively. We can now either use (Riemannian) Morse theory on the path space of $S$ or Lorentzian morse theory on the path space of $X$ to arrive at the following conclusion:

\emph{
$P^{\mathrm{t/c}}(X)\bigl((0,p),(T,q) \bigr)$ is homotopy equivalent to a finite CW-complex with a cell of dimension $\lambda$ for each geodesic of index $\lambda$ and length $<T$ or $\leq T$, respectively, in $S$ from $p$ to $q$.
}

\begin{figure}[tb]
\includegraphics[width=0.5\linewidth]{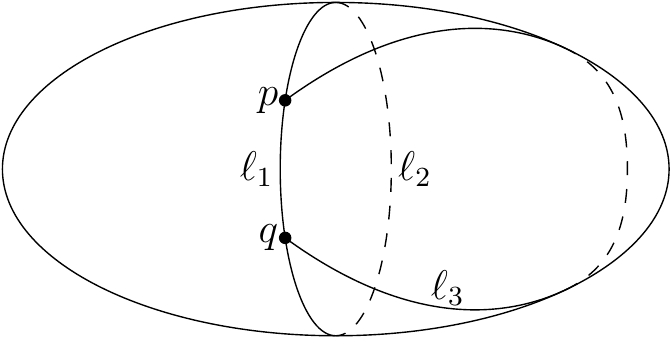}
\caption{The ellipsoid $(S,g)$ in example~\ref{ex:Non-Hausdorff} with geodesics of length $\ell_1=\frac12 \pi,\ell_2=\frac32 \pi,\ell_3\approx 1.6 \pi$.}
\label{fig:ellipsoid}
\end{figure}

As a concrete example, let $S=\set[(x,y,z)\in\RR^3]{x^2+y^2+(z/2)^2=1}$ be an elongated rotational ellipsoid with the standard metric induced from $\RR^3$.
The points $p=(1,0,0), q=(0,1,0)\in S$ decompose the equator into two geodesics $c_1$ and $c_2$ of length $\ell_1 = \frac12 \pi$ and $\ell_2=\frac32 \pi$, respectively, as seen in figure~\ref{fig:ellipsoid}. In fact, $c_1$ is the unique shortest geodesic from $p$ to $q$, and $c_2$ is the unique second-shortest geodesic. One can show that $c_1$ and $c_2$ are the only geodesics of index $0$ from $p$ to $q$. Intuitively, index $0$ means that $c_2$ is shortest among curves from $p$ to $q$ which are \enquote{sufficiently close} to $c_2$.

The aforementioned CW-complex is empty for $T=0$. As we raise $T$, we have to add a $0$-cell (i.e.\ a point) at $T=\ell_1$, another $0$-cell at $T=\ell_2$, and then eventually glue some $1$-cells (i.e.\ intervals; there will be no higher dimensional cells, as $S$ is two-dimensional) to one or both of the $0$-cells. Attaching cells of nonzero dimension never raises the number of connected components, but at some $T=\ell_3>\ell_2$, we will attach a 1-cell that connects both $0$-cells. This is because $S$ is simply connected, so its path space is connected.
The first connecting cell corresponds to a geodesic of length $\ell_3$, as shown in figure~\ref{fig:ellipsoid}.

As mentioned before, the timelike/causal homotopy classes of paths in $X$ are exactly the connected components of $P^{\mathrm{t/c}}(X)\bigl((0,p),(T,q) \bigr)$, and the number of connected components is the same as in the $CW$-complex. Therefore we get the following numbers:

{\allowdisplaybreaks
\begin{align*}
\bigl\vert \Pi^{\mathrm{t}}(X)\bigl((0,p),(T,q) \bigr) \bigr\vert
&=\begin{cases}
0 & \text{for } \phantom{\ell_0 <{}}  T \leq \ell_1 \\
1 & \text{for } \ell_1 < T \leq \ell_2 \\
2 & \text{for } \ell_2 < T \leq \ell_3 \\
1 & \text{for } \ell_3 < T 
\end{cases}
\\
\bigl\vert \Pi^{\mathrm{c}}(X)\bigl((0,p),(T,q) \bigr) \bigr\vert
&=\begin{cases}
0 & \text{for } \phantom{\ell_0 <{}} T < \ell_1 \\
1 & \text{for } \ell_1 \leq T < \ell_2 \\
2 & \text{for } \ell_2 \leq T < \ell_3 \\
1 & \text{for } \ell_3 \leq T 
\end{cases}
\\
\bigl\vert \Pi(X)\bigl((0,p),(T,q) \bigr) \bigr\vert
&=1
\end{align*}
}%
The last equation reflects the fact that $S\times \RR$ is simply connected.
Note that $\bigl\vert \Pi^{\mathrm{t}}(X)\bigl((0,p),(T,q) \bigr) \bigr\vert$ and $\bigl\vert \Pi^{\mathrm{c}}(X)\bigl((0,p),(T,q) \bigr) \bigr\vert$ differ for $T = \ell_1, \ell_2, \ell_3$.

In this example, several interesting observations can be made:%
\begin{itemize}
\item
$\Pi^{\mathrm{t}}(X)$ is not Hausdorff. Let $U_1$ and $U_2$ be open neighborhoods of the two timelike homotopy classes in $\Pi^{\mathrm{t}}(X)\bigl((0,p),(\ell_3,q) \bigr)$. By theorem~\ref{thm:Topology of Pchron}, the images of these sets under the start/endpoint map are open neighborhoods of $\bigl((0,p),(\ell_3,q) \bigr)\in X\times X$. Both of these contain $\bigl((0,p),(\ell_3+\varepsilon,q)\bigr)$ for some small $\varepsilon>0$. But since the latter element has only one preimage in $\Pi^{\mathrm{t}}(X)$, the neighborhoods $U_1$ and $U_2$ cannot be disjoint.
\item
The start/endpoint-map of $\Pi^\mathrm{c}(X)$ is not a local homeomorphism. If it were one, the preimage of the one-dimensional submanifold $\RR\cong \left\{ \bigl((0,p),(T,q)\bigr) \,\middle|\, T\in \RR \,\right\}\subset X\times X$ would be locally homeomorphic to $\RR$. One can check that this is not the case; the preimage is actually homeomorphic to three intervals which are glued together at the unique morphism in $\Pi^{\mathrm{c}}(X)\bigl((0,p),(\ell_3,q) \bigr)$.
\item
The functors $\Pi^{\mathrm{t}}(X) \to \Pi^{\mathrm{c}}(X)$ and $\Pi^{\mathrm{c}}(X) \to \Pi(X)$ induced by the inclusions $P^\mathrm{t}(X) \subseteq P^\mathrm{c}(X) \subseteq P(X)$ are not faithful (i.e.\ not injective on the morphisms). This can be seen by counting the homotopy classes for $T=\ell_3$ or $T=\ell_2$, respectively.

In other words: There are timelike paths which are causally homotopic, but not timelike homotopic to each other.

\item In general, morphisms in $\Pi^{\mathrm{t/c}}(X)$ are neither epimorphisms nor monomorphisms. For example, take $\ell_2<T<\ell_3<\tilde T$ and morphisms 
\begin{align*}
[a], [a'] &\in \Pi^{\mathrm{t/c}}(X)\bigl((0,p),(T,q) \bigr) \quad\text{with } [a]\neq [a'], \\
[b] &\in \Pi^{\mathrm{t/c}}(X)\bigl((T,p),(\tilde T,q) \bigr).
\end{align*}
Then, we have $[b][a]=[b][a']$ despite $[a]\neq [a']$, because there is only one homotopy class in $\Pi^{\mathrm{t/c}}(X)\bigl((0,p),(\tilde T,q) \bigr)$. This is in contrast to the fundamental groupoid $\Pi(X)$, in which every morphism is an isomorphism.
\end{itemize}

\end{example}

\section*{Acknowledgements}
The author acknowledges funding by the Deutsche Forschungsgemeinschaft (DFG, German Research Foundation) – 281869850 (RTG 2229), thanks Wilderich Tuschmann and Olaf Müller for interesting discussions and many helpful comments on drafts of this paper, and Marius Graeber for the idea behind example~\ref{ex:Non-Hausdorff}.
The definition of the category $\Pi^{\mathrm{c}}$ is inspired by, but not equivalent to the definition of the fundamental category of a directed space in \cite{grandis_2009}. 

\printbibliography

\section*{Appendix}

\begin{proof}[Proof of lemma~\ref{lem:Topology of fundamental groupoid}]
A topological space with the given properties has a universal covering $p \colon \tilde X \to X$. We first look at the commuting diagram
\[
\begin{tikzcd}
\Pi(\tilde X) \arrow[swap]{d}{(s,e)_{\Pi(\tilde X)}} \arrow{r}{p_\ast} & \Pi(X) \arrow{d}{(s,e)_{\Pi(X)}}
& {[\tilde c]} \arrow[mapsto]{d}\arrow[mapsto]{r} & {[p\circ \tilde c]} \arrow[mapsto]{d} \\
\tilde X\times \tilde X \arrow{r}{(p,p)} & X \times X
& \bigl(\tilde c(0),\tilde c(1) \bigr) \arrow[mapsto]{r} & \bigl(p(\tilde c(0)),p(\tilde c(1)) \bigr) .
\end{tikzcd}
\]
It is not hard to see that all of these maps are open and continuous.
Since $\tilde X$ is simply connected, the homotopy class (relative to start and endpoints) of a path in $\tilde X$ is uniquely given by its start and endpoint. This shows that $(s,e)_{\Pi(\tilde X)}$ is a homeomorphism, and $\Pi(\tilde X)$ is therefore simply connected. 

As the map $(p,p)$ is a universal covering, and $(s,e)_{\Pi(\tilde X)}$ is a homeomorphism, the map $(p,p)\circ(s,e)_{\Pi(\tilde X)} = (s,e)_{\Pi(X)}\circ p_\ast$ is also a universal covering.

We will now show that $p_\ast$ itself is also a covering: Let $[c]\in \Pi(X)(x,y)$ and let $U\subset X$ be an open neighborhood of $x$ over which $p$ is trivial. 
 This means that the preimage $p^{-1}(U)$ decomposes into disjoint open sets $U_i$, such that $p\vert_{U_i}\colon U_i \to U$ is a homeomorphism for every $i$. 

The set $\mathcal{V}\coloneqq \set[{[c']\in \Pi(X)}]{c'(0)\in U}$ is an open neighborhood of $[c]$ in $\Pi(X)$. Its preimage $p_\ast^{-1}(\mathcal{V})$ decomposes into disjoint open sets  $\mathcal{V}_i \coloneqq \set[{[\tilde c]\in \Pi(\tilde X)}]{\tilde c (0)\in U_i}$. By the homotopy lifting property of the universal covering, there is exactly one preimage (lifting) of every path class $[c']\in \mathcal{V}$ in every $\mathcal{V}_i$. Therefore, $p_\ast$ maps all $\mathcal{V}_i$ homeomorphically to $\mathcal{V}$, hence it is a covering.

We have just shown that both $(s,e)_{\Pi(X)} \circ p_\ast$ and $p_\ast$ are universal coverings, which implies that $(s,e)_{\Pi(X)}$ is also a covering.
The deck transformation group of $(s,e)_{\Pi(X)} \circ p_\ast$ is isomorphic to $\pi_1(X)\times \pi_1(X)$, where the first factor acts on the starting points and the second factor acts on the endpoints of path classes. The deck transformation group of $p_\ast$ is isomorphic to $\pi_1(X)$ and it acts on starting points and endpoints simultaneously, i.e.\ it acts as the diagonal subgroup $\Delta \pi_1(X) \subseteq \pi_1(X)\times \pi_1(X)$. Therefore, the diagram
\[
\begin{tikzcd}
\Pi(\tilde X) \arrow{d}{(s,e)_{\Pi(\tilde X)}} \arrow{r}{p_\ast} & \Pi(X) \arrow{d}\arrow{r}{(s,e)_{\Pi(X)}} & X \times X \arrow{d} \\
\tilde X\times \tilde X \arrow{r} & \quotient{(\tilde X\times \tilde X)}{\Delta\pi_1(X)} \arrow{r} & \quotient{(\tilde X\times \tilde X )}{\bigl( \pi_1(X)\times \pi_1(X) \bigr)}
\end{tikzcd}
\]
commutes, the horizontal arrows are coverings and the vertical arrows are homeomorphisms.

From this diagram, we can work out the homeomorphisms
\[ e^{-1}(x_0) = \bigl((s,e)_{\Pi(X)}\bigr)^{-1} \bigl(X\times \{\,x_0\,\} \bigr) \cong \quotient{\bigl(\tilde X \times p^{-1}(x_0) \bigr)}{\Delta \pi_1(X)} \cong \tilde X \] for any point $x_0\in X$, which implies ii). 

The statements iii) and iv) are a direct consequence of the fact that $(s,e)_{\Pi(X)}$ is a covering.

\end{proof}

\end{document}